\documentclass[11pt,a4paper]{article}
\usepackage{geometry}

\usepackage[utf8]{inputenc}

\usepackage{amsmath,amsfonts,amssymb,amsthm}
\usepackage[alphabetic]{amsrefs}
\usepackage{mathtools}
\usepackage[only,llbracket,rrbracket]{stmaryrd}
\usepackage{dsfont}

\usepackage[textsize=footnotesize,textwidth=20ex,colorinlistoftodos]{todonotes}

\usepackage{enumitem}
    \setenumerate[0]{label={\rm (\roman*)}, leftmargin=*}

\usepackage{hyperref}

\usepackage[capitalize]{cleveref}

\numberwithin{equation}{section}
\newtheorem{lemma}[equation]{Lemma}
\newtheorem{proposition}[equation]{Proposition}
\newtheorem{corollary}[equation]{Corollary}
\newtheorem{theorem}[equation]{Theorem}
\newtheorem*{maintheorem}{Main Theorem}

\theoremstyle{definition}
\newtheorem{definition}[equation]{Definition}
\newtheorem{notation}[equation]{Notation}

\newtheorem{remark}[equation]{Remark}
\newtheorem{assumption}[equation]{Assumption}

\AddToHook{env/lemma/begin}{\crefalias{equation}{lemma}}
\AddToHook{env/proposition/begin}{\crefalias{equation}{proposition}}
\AddToHook{env/corollary/begin}{\crefalias{equation}{corollary}}
\AddToHook{env/theorem/begin}{\crefalias{equation}{theorem}}
\AddToHook{env/maintheorem/begin}{\crefalias{equation}{maintheorem}}

\AddToHook{env/definition/begin}{\crefalias{equation}{definition}}
\AddToHook{env/notation/begin}{\crefalias{equation}{notation}}
\AddToHook{env/example/begin}{\crefalias{equation}{example}}
\AddToHook{env/remark/begin}{\crefalias{equation}{remark}}
\AddToHook{env/assumption/begin}{\crefalias{equation}{assumption}}

\newcommand{\1}{_{-1}}
\newcommand{\2}{_{-2}}
\newcommand{\vacuum}{\mathds{1}}

\newcommand{\ZZ}{\mathbb{Z}}
\newcommand{\NN}{\mathbb{N}}
\newcommand{\D}{\mathcal D}

\newcommand{\gLie}{\mathfrak{g}}
\newcommand{\gAffLie}{\widehat{\gLie}}
\newcommand{\baseField}{k}
\newcommand{\symsq}{\operatorname{S^2}}
\newcommand{\indexset}{\mathbf I}
\newcommand{\twosym}{_{(2)}^{\mathrm{sym}}}

\newcommand{\crefiterateformula}{\cref{def:vertex_algebra}\cref{item:iterate_formula}}

\newcommand\restrict[1]{\raisebox{-.2ex}{$|$}_{#1}}
\newcommand{\dash}{\nobreakdash-\hspace*{0pt}}

\newcommand{\thetaAlt}{\theta} 
\newcommand{\bulletAlt}{\bullet} 
\newcommand{\Talt}{T} 
\newcommand{\tauAlt}{\tau} 
\newcommand{\Id}{\mathrm{Id}}

\DeclareMathOperator{\Char}{char}
\DeclareMathOperator{\End}{End}
\DeclareMathOperator{\im}{im}

\DeclareMathOperator{\ad}{ad}

\crefname{enumi}{}{}
\Crefname{enumi}{Item}{Items}
\crefname{equation}{}{}
\Crefname{equation}{Equation}{Equations}
\crefname{notation}{Notation}{Notations}
\crefname{maintheorem}{Theorem}{Theorems}

\newcommand{\lc}{\emph{loc.\@ cit.\@}}

\begin{document}

\title{Constructing Chayet--Garibaldi algebras from affine vertex algebras (including the $3876$-dimensional algebra for $E_8$)}

\author{Tom De Medts \and Louis Olyslager}

\maketitle

\begin{abstract}
	In 2021, Maurice Chayet and Skip Garibaldi provided an explicit construction of a commutative non-associative algebra on the second smallest representation of $E_8$ (of dimension $3875$) adjoined with a unit. In fact, they define such an algebra $A(\gLie)$ for each simple Lie algebra $\gLie$, in terms of explicit but ad-hoc formulas.
	
	We discovered that their algebras $A(\gLie)$ have a natural interpretation in terms of affine vertex algebras, and their ad-hoc formulas take an extremely simple form in this new interpretation.
	It is our hope that this point of view will lead to a better understanding of this interesting class of algebras.
\end{abstract}


\section{Introduction}

The second smallest irreducible representation of algebraic groups of type~$E_8$, which is of dimension $3875$, has received some attention lately, because it admits a (unique) invariant commutative non-associative algebra product.
This fact itself is not surprising, but it was not until a few years ago that this algebra has been constructed explicitly.
This had been achieved almost simultaneously by the first author and Michiel Van Couwenberghe \cite{DMVC21} and by Maurice Chayet and Skip Garibaldi \cite{CG21}. In fact, in both constructions, the algebra is a $3876$-dimensional algebra obtained by adjoining a unit, and the construction relies on the symmetric square of the Lie algebra.
The fact that the resulting algebra has as its automorphism group precisely the original group of type $E_8$ (and nothing more) is a consequence of \cite{GG15}.

The construction of Chayet and Garibaldi from \cite{CG21} is quite general. In fact, it takes as input
\begin{itemize}[itemsep=0ex]
	\item \emph{any} Lie algebra $\gLie$ over any field $k$ with $\Char(k) \neq 2$,
	\item a non-degenerate invariant symmetric bilinear form $\kappa$ on $\gLie$,
\end{itemize}
and produces a commutative non-associative algebra that we will denote by $A(\gLie, \kappa)$.
Chayet and Garibaldi only considered the case where $\gLie$ is a simple Lie algebra, in which case $\kappa$ is necessarily a non-zero multiple of the Killing form $K$, and they only use the fixed scaling $\kappa = \frac{1}{2h^\vee} K$, where $h^\vee$ is the dual Coxeter number of $\gLie$.
Their algebra is simply denoted by $A(\gLie)$.

\bigskip

The explicit formulas in the construction of Chayet and Garibaldi are very much \emph{ad hoc}---see \cref{subsection:cg_algebras} below---and were discovered by ``trial and error''.
The construction in \cite{DMVC21} is slightly different and builds the algebra piece by piece, and therefore gives perhaps even less insight in the algebraic structure.

It was an unexpected and pleasant surprise to us that we recognized the formulas of \cite{CG21} in a quite unrelated area, namely in the theory of \emph{vertex (operator) algebras}.
More precisely, we will be able to construct the algebras $A(\gLie, \kappa)$ as substructures of a simple vertex algebra $L_{\gAffLie}(1,0)$ arising from%
\footnote{We point out that it is common to define vertex algebras $L_{\gAffLie}(\ell,0)$ ``of level $\ell$'' for any constant $\ell$; see, for example, \cite[\S 6.2]{LL04}.
	We have decided to incorporate this constant in the construction of $\gAffLie$ instead, because the resulting formulas are easier while we do not lose any generality.
	(Notice, in particular, that we do \emph{not} normalize the bilinear form as in \cite[Remark 6.2.15]{LL04}.) See also \cref{remark:level} below.}
the affine Lie algebra $\gAffLie$ corresponding to~$\gLie$.

Vertex algebras have been introduced by Richard Borcherds in 1986 (see~\cite{Bor86}) and have famously played a key role in the monstrous moonshine.
We follow Borcherds' original approach and we allow for an arbitrary base field $k$ with $\Char(k) \neq 2$.
For the concrete setup, we will essentially follow \cite{LL04}, but modified to allow for such arbitrary fields~$k$, as in~\cite{LM18}.


Usually, vertex algebras are defined using formal Laurent series, but we prefer to avoid this setup and work directly with the infinitely many bilinear ``multiplication maps'', for two reasons.
First, we will need to explicitly compute some of these multiplication maps anyway.
Second, this also helps to make our paper more accessible to people who are not at all familiar with the theory of vertex algebras.

%

\medskip

Our main result can be summarized as follows.
(We refer to the text for a detailed explanation of the notation used in this statement.)
\begin{maintheorem}\label{mainthm}
	Let $\gLie$ be a Lie algebra over a field $k$ with $\Char(k) \neq 2$ with $Z(\gLie) = 0$, equipped with a non-degenerate invariant symmetric bilinear form $\kappa$.
	Let~$V$ be the $\NN$\dash graded affine vertex algebra $V = V_{\gAffLie}(1,0)$ (which depends also on $\kappa$) and let $L$ be the quotient of $V$ by its unique maximal graded ideal, which is a simple graded vertex algebra.
	Let $L\twosym$ be the subspace of $L_{(2)}$ spanned by the elements $a_{-1}b + b_{-1}a$ for all $a,b \in L_{(1)}$.
	
	\smallskip
	
	Make $L\twosym$ into a commutative algebra by the ``Jordan product''
	\[ a \bullet b \coloneqq \tfrac{1}{2}(a_1b + b_1a) \]
	and define a bilinear form on $L\twosym$ by
	\[ (a, b) \mapsto a_3 b \]
	for all $a,b \in L\twosym$.
	Then $L\twosym$ is isomorphic to the Chayet--Garibaldi algebra $A(\gLie, \kappa)$ (as an algebra with bilinear form, up to scaling of the bilinear form).
\end{maintheorem}

It is our hope that this construction will provide new insight into the Chayet--Garibaldi algebras $A(\gLie, \kappa)$, including the notorious $(3875 + 1)$\dash dimensional algebra for~$E_8$.

\paragraph*{Outline of the paper}

In \cref{subsection:cg_algebras}, we review the construction of the class of algebras $A(\gLie)$ constructed by Chayet and Garibaldi, generalized to arbitrary Lie algebras equipped with a non-degenerate symmetric bilinear invariant form.

In \cref{subsection:vertex_algebras}, we give the required background on vertex algebras (in arbitrary characteristic $\neq 2$). We have, on purpose, restricted this to the bare minimum while keeping the paper self-contained.
In the next \cref{subsection:affine_vertex_algebras}, we then present the construction of affine vertex algebras (i.e., vertex algebras arising from affine Lie algebras), translated to our setup from \cref{subsection:vertex_algebras}. At the end of this section, we collect a number of useful identities.

\Cref{section:results}, finally, is the core of the paper. We combine the Chayet--Garibaldi algebras from \cref{subsection:cg_algebras} with the affine vertex algebras from \cref{subsection:affine_vertex_algebras}, culminating in our main result in \cref{theorem:main_result}.

    
\paragraph*{Acknowledgments}

We are grateful to Jari Desmet for many interesting discussions. In particular, the proof of \cref{lemma:kernels_are_equal} relies on an observation due to him, which was a crucial ingredient to prove our main result in full generality.
We thank an anonymous referee for their insightful comments. In particular, one of their pertinent questions gave rise to \cref{lemma:unit,remark:unital} dealing with the question whether the algebras $A(\gLie, \kappa)$ are unital and revealing a connection with vertex \emph{operator} algebras.

The second author is supported by the FWO PhD mandate 1105425N.

\section{The Chayet--Garibaldi algebras $A(\gLie, \kappa)$}\label{subsection:cg_algebras}

We start by summarizing the construction of the algebras $A(\gLie, \kappa)$ introduced by Chayet and Garibaldi in \cite{CG21}.
We give their construction in a more general setting than they do, by allowing for arbitrary Lie algebras equipped with a non-degenerate symmetric bilinear invariant form.
(See \cref{assumptions:lie_algebras} below for the precise setting that we use.)

\begin{definition}\label{def:Lie_bilinear_form}
	Let $k$ be a commutative field, let $\gLie$ be a Lie algebra over $k$ and let
	\[ \kappa \colon \gLie \times \gLie \to k \colon (a, b) \mapsto \langle a, b \rangle \]
	be a symmetric bilinear form on $\gLie$.
	\begin{enumerate}
		\item We call $\kappa$ \emph{(Lie or $\gLie$-) invariant} if
			$\langle [a,b] , c \rangle = \langle a, [b,c]\rangle$
			for all $a,b,c \in \gLie$.
		\item We call $\kappa$ \emph{non-degenerate} if the only $a\in \gLie$ for which $\langle a, \gLie \rangle = 0$ is $a = 0$.
	\end{enumerate}
\end{definition}
\begin{remark}
	For a finite-dimensional \emph{simple} Lie algebra, the \emph{Killing form} $K(a,b) \coloneqq \operatorname{Tr}(\ad_a\circ\ad_b)$ is always an invariant form. It is either identically zero or non-degenerate. The former may happen if $\Char k \not =0$ (see, e.g., \cite[p.\@~17]{S67} for a discussion). On the other hand, when $\Char k = 0$, the Killing form is always non-degenerate, by Cartan's semisimplicity criterion (see, e.g., \cite[Theorem 5.24]{H24}), and by \cite[Section 2, ``Global hypotheses'']{CG21}, this is also true when $\gLie$ is the Lie algebra of an absolutely simple linear algebraic group~$G$ with $\Char k \geq h+2$, where $h$ is the Coxeter number, a constant depending only on the type of~$G$.
	Notice, for instance, that for type $E_8$, it is known by \cite[Theorem~B]{GP09} that when $\Char(k) = 2$, $3$ or $5$, the Killing form is zero (and in fact, the trace form of any $\gLie$-representation is zero).
	(We thank an anonymous referee for pointing this out.)
\end{remark}


\begin{assumption}\label{assumptions:lie_algebras}
	We assume for the rest of this paper that
	\begin{enumerate}[itemsep=0ex]
		\item\label{assumptions:lie_algebras:characteristic} $\baseField$ is a field of characteristic different from $2$;
		\item $\gLie$ is a Lie algebra over $\baseField$ with trivial center $Z(\gLie) = 0$;
		\item\label{assumptions:lie_algebras:form} $\kappa = \langle\cdot,\cdot\rangle$ is a non-degenerate symmetric invariant bilinear form on $\gLie$.
	\end{enumerate}
\end{assumption}
The condition $Z(\gLie) = 0$ will only be used much later, in the proof of \cref{lemma:kernels_are_equal}.


\begin{definition}\label{definition:symmetric_map}
	Let $W$ be a vector space, let $M \colon W \to W$ be a linear map and let $f$ be a non-degenerate bilinear form on $W$.
	We call $M$ \emph{symmetric} (with respect to~$f$) if $f(Mv,w)=f(v,Mw)$ for all $v,w \in W$, and we call $M$ \emph{skew-symmetric} if $f(Mv,w)=-f(v,Mw)$ for all $v,w \in W$.
\end{definition}


\begin{definition}\label{definition:symmetric_square}
	The \emph{symmetric square} of $\gLie$ is
	\[\symsq \gLie \coloneqq \gLie\otimes \gLie / \langle a\otimes b - b\otimes a \mid a,b\in \gLie\rangle.\]
	We use the notation $ab$ to denote the element $a\otimes b$ of $\symsq \gLie$. 
\end{definition}
Notice that $(a+b)(a+b) = aa + bb + 2ab$ for all $a,b\in \gLie$, hence $\symsq \gLie$ is spanned by the ``symmetric'' elements $aa$, for $a\in \gLie$.
(Recall that $\Char(k) \neq 2$ by \cref{assumptions:lie_algebras}\ref{assumptions:lie_algebras:characteristic}).
We use this observation repeatedly in the following definition.

\begin{definition}\label{definition:Ag}
	\begin{enumerate}
		\item 
			We define a commutative product on $\symsq \gLie$ by setting
			\begin{equation}\label{formula:star}
				aa * bb \coloneqq  
				a\bigl[b,[b,a]\bigr]+b\bigl[a,[a,b]\bigr]+[a,b][a,b]
				+2\langle a,b\rangle ab
			\end{equation}
			for all $a,b\in\gLie$ and extending linearly.
		\item
			We define a map $S \colon \symsq\gLie \to \End(\gLie)$ by setting
			\begin{equation}\label{def:chga_S}
				S(aa) \coloneqq \bigl[a,[a,\cdot]\bigr] + 2\langle a,\cdot\rangle a
			\end{equation}
			for all $a \in \gLie$ and extending linearly.
		\item 
			We define a commutative product $\diamond$ on $\im S$ by setting
			\begin{equation*}
				S(aa)\diamond S(bb) \coloneqq S(aa * bb)
			\end{equation*}
			for all $a,b\in\gLie$ and extending linearly.

			By~\cite[Lemma 4.2]{CG21}, this product is well-defined%
			\footnote{The proof of~\cite[Lemma 4.2]{CG21} holds \emph{mutatis mutandis} in our setting. More precisely, it suffices to replace the symbols `$h^\vee$' by `$1$', `$K$' by `$2\langle\cdot,\cdot\rangle$', and `$\sum P(X_i^2)$' by `$\sum 2\langle X_i,\cdot \rangle X_i$'.}.
			We denote the resulting algebra $(\im S, \diamond)$ by $A(\gLie, \kappa)$.
	\end{enumerate}
\end{definition}
%
%

\begin{remark}\label{remark:coxeter_number}
	\Cref{definition:Ag} first appeared in \cite[(4.1)]{CG21} in the specific case that $\gLie$ is
the Lie algebra of an absolutely simple linear algebraic group $G$ over $\baseField$ with $\Char(\baseField) \geq h+2 > h^\vee>0$ or $\Char k = 0$, where $h$ is the Coxeter number and $h^\vee$ is the dual Coxeter number of the Lie algebra~$\gLie$. In particular, such $\gLie$ is a simple Lie algebra.
	Chayet and Garibaldi use the specific non-degenerate bilinear form
	\[ \kappa \coloneqq \frac{1}{2h^\vee} K , \]
	where $K$ is the Killing form on $\gLie$.
	Notice that compared to their setup, we have split up the definition in two steps by first defining a product $*$ on $\symsq\gLie$ and then transforming this product to $\im S$.
	The proof of \cite[Lemma 4.2]{CG21} then amounts to verifying that the kernel of $S$ is an ideal for the product~$*$.
\end{remark}
\begin{lemma}\label{le:CG ids}
	Let $a,b \in \gLie$ and $f \in \im S$. Then
	\begin{enumerate}
		\item\label{item:adbSaa} $[\ad_b, S(aa)] = 2 S(a [b,a])$;
		\item\label{item:f diamond Sbb} $f \diamond S(bb) = \tfrac{1}{2} [\ad_b, [\ad_b, f]] + S(f(b) b)$.
	\end{enumerate}
\end{lemma}
\begin{proof}
	\begin{enumerate}
		\item This is the $\gLie$-equivariance of $S$; see \cite[(4.4)]{CG21}. Alternatively, this can be computed directly from the defining equation \eqref{def:chga_S}.
		\item Since $\im S$ is spanned by elements of the form $S(aa)$ with $a \in \gLie$, it suffices to verify the claim for $f = S(aa)$. Just as in the proof of \cite[Lemma 4.5]{CG21}, it follows from \cref{item:adbSaa} and \cref{definition:Ag} that
	\[
		S(aa) \diamond S(bb) = \tfrac{1}{2} \bigl[\ad_b, [\ad_b, S(aa)]\bigr] + S\bigl((\ad_{a}^2 b)b\bigr) + 2\langle  a, b \rangle S(ab).
	\]
	By \eqref{def:chga_S}, we have $S(aa)b = \ad_{a}^2 b + 2 \langle a, b \rangle a$, so we get
	\begin{equation*}
		S(aa) \diamond S(bb) = \tfrac{1}{2} \bigl[\ad_b, [\ad_b, S(aa) ]\bigr] + S\bigl((S(aa)b)\,b\bigr).
		\qedhere
	\end{equation*}
	\end{enumerate}
\end{proof}
The algebras from \cite{CG21} are always unital; see \cite[Lemma 4.5]{CG21}. In our more general setting, we have the following result.
(We thank an anonymous referee for asking this question. See also \cref{remark:unital} below.)
\begin{lemma}\label{lemma:unit}
	The algebra $A(\gLie,\kappa)$ is unital if and only if $\Id_\gLie \in \im S$, and then $\Id_\gLie$ is the unit element.
\end{lemma}
\begin{proof}
	Suppose first that $\Id_\gLie \in \im S$. Then it follows immediately from \cref{le:CG ids}\cref{item:f diamond Sbb} that $\Id_\gLie \diamond S(bb) = S(bb)$ for all $b \in \gLie$, so indeed $A(\gLie, \kappa)$ is unital, with unit element $\Id_\gLie$.
	
	Conversely, suppose that $A(\gLie,\kappa)$ is unital and let $f \in \im S$ be the unit element of $A(\gLie,\kappa)$.
	Write $g \coloneqq f - \Id_\gLie$ (as operator on $\gLie$).
	By \cref{le:CG ids}\cref{item:f diamond Sbb} again, we have
	\[ S(bb) = f \diamond S(bb) = \tfrac{1}{2} [\ad_b, [\ad_b, f]] + S(f(b) b) \]
	for all $b \in \gLie$, which we rewrite as
	\[ S(g(b)b) + \tfrac{1}{2} [\ad_b, [\ad_b, g]] = 0 . \]
	Applying this to the element $b$ gives
	\[ S(g(b)b) b + \tfrac{1}{2} [b, [b, g(b)]] = 0 . \]
	On the other hand, we have $S(g(b)b) b = \tfrac{1}{2}[b, [g(b), b]] + \langle g(b), b \rangle b + \langle b, b \rangle g(b)$ by \eqref{def:chga_S}, so we get
	\begin{equation}\label{eq:gb b}
		\langle g(b), b \rangle b + \langle b, b \rangle g(b) = 0
	\end{equation}
	for all $b \in \gLie$.
	
	As $\kappa$ is symmetric and non-degenerate, we can find an orthogonal basis $\mathcal{B}$ for $\gLie$.
	For any $b \in \mathcal{B}$, we then have $\langle b,b \rangle \neq 0$ and hence, by~\eqref{eq:gb b}, $g(b) = \lambda_b b$ for some $\lambda_b \in \baseField$. However, \eqref{eq:gb b} then implies $2\lambda_b \langle b, b \rangle b = 0$ and thus $\lambda_b = 0$ for all $b \in \mathcal{B}$. It follows that $g(b) = 0$ for all $b \in \mathcal{B}$, so we conclude that $g=0$ and hence $f = \Id_\gLie$; in particular, $\Id_\gLie \in \im S$.
\end{proof}
\begin{remark}
	Let $\gLie$ be finite-dimensional and let $\mathcal{B} = \{ u_1,\dots,u_d \}$ be an orthonormal basis of $\gLie$ for the form $\kappa$.
	Assume that the Casimir element $\Omega := \sum_{i=1}^d u_iu_i \in S^2\gLie$ (with respect to $\kappa$) acts on $\gLie$ as a scalar $2h$, i.e.,
	\[ \sum_{i=1}^d \bigl[u_i[u_i,a]\bigr] = 2h a \]
	for all $a\in \gLie$. If $h \neq -1$, then $S(\Omega)$ is a non-zero multiple of $\Id_\gLie$.
	(Indeed, it follows from~\eqref{def:chga_S} with $a = u_i$ and summing over all $i$ that $S(\Omega) = 2(h+1)\Id_\gLie$.)
	
	Notice that this is exactly the situation described in \cite[p.\@~209]{LL04}. Following the two examples from \cite[Remark 6.2.15]{LL04}, we see that the algebra $A(\gLie, \kappa)$ is unital in the following two situations:
	\begin{enumerate}[itemsep=0ex]
		\item If $\gLie$ is an abelian Lie algebra, then $\Omega$ acts on $\gLie$ as $0$ and hence $A(\gLie,\kappa)$ is unital.
		\item If $\gLie$ is simple, then $\Omega$ acts on $\gLie$ as a scalar $2h$; if $h \neq -1$, then again $A(\gLie,\kappa)$ is unital.
	\end{enumerate}
\end{remark}

\medskip


The algebra $A(\gLie, \kappa)$ comes equipped with a ``natural'' bilinear form.
\begin{definition}\label{definition:tau}
	We define a bilinear form $\tau$ on the algebra $A(\gLie, \kappa)$ by setting
	\begin{equation*}
		\tauAlt(S(aa),S(bb)) = \tfrac12\langle S(aa)b,b\rangle
	\end{equation*}
	for all $a,b \in \gLie$ and extending linearly.
	(It is not hard to verify that this is well-defined.)
	This is a scaled version of the form defined in~\cite[Example 5.1]{CG21}.
	By \cite[Lemma~6.1]{CG21}, this form is associative, in the sense that
	\begin{equation*}
		\tauAlt(a\diamond b,c) = \tauAlt(a,b\diamond c)
	\end{equation*}
	for all $a,b,c\in A(\gLie,\kappa)$.
\end{definition}

\begin{remark}
	In \cite[Definition 1.1]{DMVC21}, we have called an algebra equipped with a non-degenerate associative bilinear form a \emph{Frobenius algebra}. More precisely, the algebra $A(\gLie, \kappa)$ is a \emph{Frobenius algebra for $\gLie$} in the sense that both the product and the bilinear form are $\gLie$-invariant.
	In fact, our main theorem on page~\pageref{mainthm} can be restated as an isomorphism between two Frobenius algebras. Compared to the original definition from \cite{CG21}, this isomorphism only holds up to scaling of the bilinear form, which is why we have added this phrase in our Main Theorem.
\end{remark}
\section{Vertex algebras}\label{subsection:vertex_algebras}

Vertex algebras have been introduced by Richard Borcherds in \cite{Bor86}.
We will only need vertex algebras arising from affine Lie algebras, which is a well understood class of vertex algebras;
we will recall their construction in \cref{subsection:affine_vertex_algebras}.
We emphasize that unlike \cite{LL04}, we allow a base field of arbitrary characteristic, following the approach from \cite{LM18}.
Our setup is perhaps somewhat atypical, avoiding the formal Laurent series framework, to make everything as accessible as possible to non-experts. See also \cref{rem:VAdef} below.



As in \cite[Section 2]{LM18}, we will use the obvious interpretation of binomial coefficients in arbitrary fields, i.e., for all $m \in \ZZ$ and $n \in \NN$, we first define
	\[ \binom m n \coloneqq \frac{m(m-1)\dotsm (m+1-n)}{n!} \in \ZZ \]
	and view it as an element of $k$ via the unique ring homomorphism $\ZZ \to k$.
	Notice, in particular, that $\binom{-1}{n} = (-1)^n$ for all $n \in \NN$.
\begin{definition}[{\cite[Section 4]{Bor86}; see also \cite[Definition 2.1]{LM18}}]\label{def:vertex_algebra}
	Let $V$ be a vector space and let $\vacuum \in V$ be a distinguished non-zero element called the \emph{vacuum vector}.
	For each integer $n\in\ZZ$, define a product map
	\[ V \times V \to V \colon (a,b) \mapsto a_n b . \]
	We call this the \emph{$n$-th product}. We use the notation
	\[ a_n \colon V \to V \colon b \mapsto a_n b \]
	for the left multiplication operators. Let $\mathcal{D}^{(m)}$ be a linear operator on $V$ for each $m\in \ZZ$. Then $V$ is a \emph{vertex algebra} if it satisfies the following axioms, for all $a,b,c\in V$ and all $n,m\in \ZZ$:
	\begin{enumerate}
		\item $a_nb=0$ for $n$ sufficiently large.\hfill  (\emph{truncation condition})\label{item:truncation_condition}  
		\item $\vacuum_{n} a = \delta_{-1,n}a$.\hfill  (\emph{vacuum property})\label{item:vacuum_property}
		\item $a_n\vacuum = \mathcal{D}^{(-n-1)}a$. \hfill (\emph{$\D$-derivate formula})\label{item:D_derivate_formula} 
		\item $a_nb=\sum_{i\geq 0}(-1)^{i+n+1}\mathcal D^{(i)}b_{n+i}a$. \hfill (\emph{skew symmetry})\label{item:skew_symmetry}
		\item $(a_mb)_nc = \sum_{i=0}^\infty (-1)^i\binom m i\left(a_{m-i}b_{n+i}-(-1)^mb_{n+m-i}a_i\right)c$. \hfill (\emph{iterate formula})\label{item:iterate_formula}
	\end{enumerate} 
\end{definition}
\begin{remark}\label{rem:VAdef}
	There are many equivalent definitions of a vertex algebra. For example, the iterate formula \ref{item:iterate_formula} and skew symmetry \ref{item:skew_symmetry} are equivalent, under the assumptions of the other axioms of a vertex algebra, to \emph{Borcherds' identity}:
	\begin{equation}\label{Borcherds identity}
				\sum_{i=0}^\infty\binom{m}{i}\left(a_{p+i}b\right)_{m+n-i}c = \sum_{i=0}^\infty(-1)^i\binom{p}{i}\left(a_{m+p-i}b_{n+i}-(-1)^pb_{n+p-i}a_{m+i}\right)c.
			\end{equation}
	Vertex algebras are usually introduced using formal Laurent series. (See, for example, \cite[Definition 3.1.1]{LL04} and \cite[Definition 2.1]{JLM19}.) This allows to combine the infinitely many products in one function $Y \colon V\to \End(V) \llbracket x,x^{-1} \rrbracket \colon a\mapsto \sum_{n\in\ZZ} a_nx^{-1-n}$, which therefore plays a key role in the theory.
	We have chosen to avoid this approach to make the setup more accessible to non-experts. In addition, this has the advantage that we can directly work with the products that we need in our main theorem.
	In this context, see also \cite[Proposition 3.6.6]{LL04}.
\end{remark}
\begin{notation}
We write $\D$ (without superscript) for the operator $\D^{(1)}$. 	
\end{notation}

We have the following easy consequence (see \cite[p. 90]{LL04} and \cite[Lemma 2.2]{LM18}):
\begin{lemma}Let $V$ be a vertex algebra. Then 
\begin{enumerate}
	\item $\D^{(n)} = 0$ for all $n<0$;
	\item $a_n\vacuum = 0$ for all $a\in V$ and all $n\geq 0$;
	\item $\D^{(0)} = \Id$.
\end{enumerate}
\end{lemma}
\begin{proof}
	Let $a\in V$ and $n\in\ZZ$ with $n \geq 0$. Then
	\[
		\D^{(-n-1)}a = a_n\vacuum = \sum_{i\geq 0}(-1)^{i+n+1}\mathcal D^{(i)}\vacuum_{n+i}a = \sum_{i\geq 0}\mathcal D^{(i)}\delta_{n+i,-1}a = 0,
	\]
	and
	\begin{multline*}
	a = \vacuum\1a = \sum_{i\geq 0}(-1)^{i}\mathcal D^{(i)}a_{i-1}\vacuum = \D^{(0)}a_{-1}\vacuum 
	= (a\1\vacuum)\1\vacuum \\
	= \sum_{i=0}^\infty (-1)^i\tbinom{-1} i\left(a_{-1-i}\vacuum_{-1+i}+\vacuum_{-2-i}a_i\right)\vacuum 
	= a_{-1}\vacuum_{-1}\vacuum = a_{-1}\vacuum = \D^{(0)}a.\qedhere
	\end{multline*}
\end{proof}

We have the obvious notion of ideals in vertex algebras; see \cite[Definition 3.9.7 and Remark 3.9.8]{LL04}.
\begin{definition}
	Let $V$ be a vertex algebra. 
	\begin{enumerate}
		\item A subspace $I$ of $V$ is called a \emph{left ideal} if $a_nb \in I$ for all $a\in V$, $b\in I$ and $n\in\ZZ$. Similarly, it is called a \emph{right ideal} if $a_nb \in I$ for all $a \in I$, $b \in V$ and $n \in \ZZ$.
		\item A subspace $I$ that is both a left and right ideal is called an \emph{ideal}. An ideal $I$ is called \emph{proper} if $V \neq I$.
	\end{enumerate}
\end{definition}
The following observation (over an arbitrary base field $k$) can also be found in the last paragraph of \cite[p.~279]{JLM19}.
\begin{lemma}\label{lemma:ideals}
	Let $V$ be a vertex algebra. Then the following are equivalent:
	\begin{enumerate}[label=\rm (\alph*),itemsep=0ex]
		\item\label{lemma:ideals;item:right} $I$ is a right ideal;
		\item\label{lemma:ideals;item:left} $I$ is a left ideal and $\D^{(n)}I \subseteq I$ for all $n\in\ZZ_{\geq0}$;
		\item\label{lemma:ideals;item:ideal} $I$ is an ideal.
	\end{enumerate} 
\end{lemma}
\begin{proof}
	If $I$ is a right ideal, then $\D^{(n)}I \subseteq I$ for all $n\in\ZZ_{\geq0}$, by the $\D$-derivate formula (\cref{def:vertex_algebra}\ref{item:D_derivate_formula}). If $I$ is a left or right  ideal and closed under all $\D^{(n)}$, then it is a two-sided ideal, by skew symmetry (\cref{def:vertex_algebra}\ref{item:skew_symmetry}). The other implications are trivial.
\end{proof}
We will later need the following related construction of an ideal starting from a left ideal.
\begin{lemma}\label{lemma:J_ideal}
	Let $J$ be a left ideal in a vertex algebra $V$. Then 
	\[ \bar{J} \coloneqq \sum_{i=0}^\infty \D^{(i)}J\] is a two-sided ideal.
\end{lemma}
\begin{proof}
	This is shown in the first paragraph of the proof of \cite[Lemma 2.7]{JLM19}.
	(Notice that the specific form of $J$ from the statement of that result is irrelevant for this part of the proof.)
%
\end{proof}

We have the obvious notions of simple vertex algebras, maximal ideals, ideals generated by a subset and quotients of a vertex algebra by an ideal.

The affine vertex algebras that we will introduce in \cref{subsection:affine_vertex_algebras} are \emph{graded} vertex algebras:
\begin{definition}
	Let $V$ be a vertex algebra.
	\begin{enumerate}
		\item \cite[Definition 2.6]{JLM19} Assume that $V = \bigoplus_{n\in\ZZ}V_{(n)}$ for certain subspaces $V_{(n)}$ with $\vacuum\in V_{(0)}$.
			Suppose that
			\begin{equation}\label{eq:graded}
				a_n b \in V_{(l+m-1-n)} \quad \text{ for all } a\in V_{(l)},\ b\in V_{(m)} \text{ and } l,m,n\in \ZZ .
			\end{equation}
			Then we call $V$ \emph{graded}. 
		\item An ideal $I$ of a graded vertex algebra $V$ is \emph{graded} if $I = \bigoplus_{i} V_{(i)}\cap I$.
		\item We say that a vertex algebra is \emph{simple graded} if $0$ and $V$ are the the only graded ideals.
		\item We say that a vertex algebra has \emph{type O}, if it is $\NN$-graded, i.e., $V = \bigoplus_{n\geq 0}V_{(n)}$, with $V_{(0)}$ one-dimensional.
		\item A vertex algebra is \emph{OZ (one-zero)} if it has type O and, in addition, $V_{(1)}=0$.
	\end{enumerate}
\end{definition}
\begin{remark}
	In the previous definition, the notions \emph{simple graded} and \emph{type O} are non-standard, but the notion \emph{OZ} is standard. When $V$ is a vertex \emph{operator} algebra (VOA, see \cite[Definition 3.1.22]{LL04} for a definition) and has type O, then it is called a vertex operator algebra \emph{of CFT type}; see \cite{DLMM98} for the original definition.
\end{remark}

The following observation is essentially the argument given in \cite[Remark 3.9.11]{LL04}.
\begin{lemma}\label{lemma:simple_vertex_algebra}
	A graded vertex algebra $V = \bigoplus_{n \geq 0} V_{(n)}$ of type O has a unique maximal graded ideal $M$.
	In particular, $V/M$ is a simple graded vertex algebra of type O.
\end{lemma}
\begin{proof}
	Since $V_{(0)} = k \vacuum$, any proper graded ideal is a subspace of $\bigoplus_{n>0}V_{(n)}$. Hence the sum of all proper graded ideals does not contain $\vacuum$ and is a maximal graded ideal.
\end{proof}
\begin{remark}
	When $k$ has characteristic zero and $V$ is a VOA, then by \cite[Remarks~3.9.10 and 3.9.11]{LL04},
	every ideal is graded. In particular, the quotient $V/M$ in \cref{lemma:simple_vertex_algebra} is then a simple vertex algebra.
	This is false, however, in characteristic $p > 0$. In fact, \cite{LM18} gives examples of affine vertex algebras in characteristic $p$ containing infinitely many maximal ideals.
\end{remark}

\section{Vertex algebras related to affine Lie algebras}\label{subsection:affine_vertex_algebras}

In this section, we construct a simple graded vertex algebra starting from the Lie algebra~$\gLie$ and the form $\kappa = \langle \cdot,\cdot\rangle$ satisfying \cref{assumptions:lie_algebras}. (In fact, not all of these assumptions are needed to construct the vertex algebra.) 


 We start by constructing an affine Lie algebra. Then we construct from its universal enveloping algebra a module for this affine Lie algebra. Next, we put a graded vertex algebra structure (of type O) on this module, by defining infinitely many products. The simple graded vertex algebra is then obtained by taking the quotient by its unique maximal graded ideal.
 
 Our construction essentially follows \cite[\S 3]{LM18}, except that the explicit basis in \cref{corollary:basis_for_V} is as in \cite[p.~206]{LL04}. Usually, a constant $\ell$ is involved as well in the construction of the vertex algebra, but we decide to incorporate this constant into the bilinear form $\kappa$, which simplifies the resulting formulas without losing any generality.
 
\begin{definition}
	The \emph{affine Lie algebra} corresponding to the pair $(\gLie, \kappa)$ is the Lie algebra with underlying vector space
	\[ \gAffLie \coloneqq \gLie\otimes \baseField[t,t^{-1}] \,\oplus\, \baseField\mathbf{k} \]
	and with Lie bracket given by
	\[ \left[a\otimes t^m,b\otimes t^n\right] \coloneqq \left[a,b\right]\otimes t^{m+n}+m\left\langle a,b\right\rangle \delta_{m+n,0}\mathbf{k} \]
	for all $a,b\in\gLie$ and $m,n\in\ZZ$, where $\mathbf{k}$ is a central element. We use the notation $\gAffLie_\kappa$ when we want to explicitly mention the bilinear form $\kappa = \langle \cdot, \cdot \rangle$.
	It is a common and convenient convention to write
	\[ a(n) \coloneqq a\otimes t^n \]
	for all $a\in \gLie$ and all $n\in \ZZ$. The defining equation for the Lie bracket can then be rewritten as
	\begin{equation}\label{eq:aff_lie_alg_eqn}
		\left[ a(m), b(n) \right] \coloneqq \left[a,b\right](m+n)+m\left\langle a,b\right\rangle \delta_{m+n,0}\mathbf{k} .
	\end{equation}
	As in \cite[(6.2.6)]{LL04}, we also define the Lie subalgebra
	\[ \gAffLie_{\leq 0}\coloneqq \gLie\otimes \baseField[t] \,\oplus\, \baseField\mathbf{k} = \left\langle a(n) \mid a \in \gLie, n \in \ZZ_{\geq 0} \right\rangle \oplus \langle \mathbf{k} \rangle . \]
	(Notice the potentially confusing notation $\gAffLie_{\leq 0}$, but see \cite[Remark 6.2.1]{LL04} for an explanation.)
\end{definition}
Next, we construct a $\gAffLie$-module.
\begin{definition}
	Let $U(\gAffLie)$ be the universal enveloping algebra of $\gAffLie$. 
	Let $\gAffLie_{\leq 0}$ act on $\baseField\vacuum$ (on the left) by setting
	\[ \mathbf{k}\cdot\vacuum \coloneqq \vacuum \quad \text{and} \quad a(n)\cdot \vacuum \coloneqq 0 \]
	for all $a\in \gLie$ and $n \in \ZZ_{\geq 0}$, and let $\gAffLie_{\leq 0}$ act on the right on $U(\gAffLie)$ by multiplication.
	Now define
	\[ V_{\gAffLie}(1,0) \coloneqq U(\gAffLie)\otimes_{\gAffLie_{\leq 0}} \baseField\vacuum . \]
	This is a (left) $\gAffLie$-module by multiplication (and hence it is also a left $U(\gAffLie)$-module).
\end{definition}
\begin{remark}\label{remark:level}
	This is the place where usually the constant $\ell\neq0$ comes into play, by defining
	$V_{\gAffLie}(\ell,0)$ as above, but with the equation
	\[ \mathbf{k}\cdot\vacuum \coloneqq \vacuum \quad\text{ replaced by }\quad  \mathbf{k}\cdot\vacuum \coloneqq \ell\vacuum. \]
	This constant is called the \emph{level} of the vertex algebra.
	It turns out (and is well known) that $\gAffLie_\kappa$ and $\gAffLie_{\ell\kappa}$ are isomorphic, with an explicit isomorphism given by $a(n)\mapsto a(n)$ and $\mathbf k \mapsto \ell \mathbf k$.
	Under this isomorphism, the resulting vertex algebras $V_{\gAffLie_\kappa}(\ell,0)$ and $V_{\gAffLie_{\ell\kappa}}(1,0)$ are isomorphic (as vertex algebras, cf.\@~\cref{proposition:V_is_VA}). The level only becomes relevant when $V_{\gAffLie_{\kappa}}(\ell,0)$ is studied in the context of vertex \emph{operator} algebras.
	Notice, in particular, that in \cite[Remark 6.2.15]{LL04}, the affine Lie algebra $\gAffLie$ is constructed using the scaling $\kappa\coloneqq\frac{1}{2h^\vee}K$, and the vertex operator algebra $V_{\gAffLie_{\kappa}}(1,0)$ obtained in \cite[Theorem~6.2.18]{LL04} then corresponds exactly to the setup of Chayet and Garibaldi; see \cref{remark:coxeter_number}.
 	See also \cref{remark:unital} below.
\end{remark}
In order to construct a basis for $V_{\gAffLie}(1,0)$, we use the Poincaré--Birkhoff--Witt theorem.
\begin{theorem}[PBW theorem]\label{theorem:pbw}
	Let $\mathfrak{h}$ be a Lie algebra over a field $\baseField$ with totally ordered basis $\mathcal{B}$.
	Then the universal enveloping algebra $U(\mathfrak h)$ has a basis 
	\[ \{a^{(1)}a^{(2)}\dotsm a^{(n)} \mid a^{(1)},\dots, a^{(n)}\in \mathcal{B} \text{ with } a^{(1)} \geq a^{(2)} \geq \dots \geq a^{(n)}\}.\]  
\end{theorem}
\begin{corollary}\label{corollary:basis_for_V}
	Let $a^{(i)}, i\in \indexset$ be an ordered basis for $\gLie$ (where $\indexset$ is some ordered index set).
	Let ``$\hspace*{.5ex}>\hspace*{-.2ex}$'' be the lexicographical order on $\ZZ \times \indexset$, i.e., we set
	\[ (m,i) > (n,j) \iff m > n, \text{ or } m=n \text{ and } i>j, \]
	for all $m,n\in\ZZ$ and $i,j\in\indexset$.
	Then the elements of the form
	\begin{equation}\label{eq:basis_V_aff}
		a^{(i_1)}(-m_1)\dotsm a^{(i_r)}(-m_r)\vacuum\in V_{\gAffLie}(1,0)
	\end{equation}
	with $(m_1,i_1) \geq \dots \geq (m_r,i_r) \in \ZZ_{> 0}\times\indexset$, and $r\in \ZZ_{\geq0}$, form a basis for $V_{\gAffLie}(1,0)$.
\end{corollary}
\begin{definition}\label{definition:Z_grading_of_V}
	Fix a basis for $\gLie$ and for $V_{\gAffLie}(1,0)$ as in \cref{corollary:basis_for_V}. For each $n \in \ZZ$, we let
	\[ V_{\gAffLie}(1,0)_{(n)} \]
	be the subspace spanned by the basis elements \eqref{eq:basis_V_aff} with $m_1+\dots+m_r=n$.
	In particular, $V_{\gAffLie}(1,0)_{(n)} =0$ for $n<0$. We call $n$ the \emph{degree} of this subspace.

	Note that this gives $V_{\gAffLie}(1,0)$ the structure of a $\ZZ$-graded vector space
	\[ V_{\gAffLie}(1,0)=\bigoplus_{n\in\ZZ} V_{\gAffLie}(1,0)_{(n)} . \]
\end{definition}
\begin{remark}\label{remark:gLie_and_V_1}
	Observe that $\gLie \cong V_{\gAffLie}(1,0)_{(1)}$ as vector spaces, via the isomorphism
	\[ a \mapsto a(-1)\vacuum . \]
\end{remark}
We can now invoke \cite[Proposition 3.3]{LM18}, translated to our setup.
\begin{proposition}\label{proposition:V_is_VA}
	There exists a unique vertex algebra structure on $V_{\gAffLie}(1,0)$ determined by the conditions that $\vacuum \in V_{\gAffLie}(1,0)_{(0)}$ is the vacuum vector and that
	\[ (a(-1)\vacuum)_nb = a(n)b \qquad \text{for all } a\in\gLie, b\in V_{\gAffLie}(1,0)\text{ and } n\in\ZZ . \]
\end{proposition}
	
We use the vector space isomorphism from \cref{remark:gLie_and_V_1} to make some identifications.
\begin{notation}\label{notation:identify_gLie_and_V}
	Using the identification from \cref{remark:gLie_and_V_1} together with \cref{proposition:V_is_VA}, we have $a_n b = a(n) b$ for all $a \in \gLie$ and all $n \in \ZZ$.
	By repeatedly applying this, we get
	\begin{equation}\label{eq:identify_gLie_and_V}
		a^{(1)}_{-m_1}\dotsm a^{(r)}_{-m_r}b = a^{(1)}(-m_1)\dotsm a^{(r)}(-m_r)b
	\end{equation}
	for all $a^{(1)},\dots,a^{(r)} \in \gLie$, $b \in V_{\gAffLie}(1,0)$ and $m_1,\dots m_r\in\ZZ$.
	In particular, we can rewrite the elements of the form \eqref{eq:basis_V_aff}, which implies that we can write each element of $V_{\gAffLie}(1,0)$ as a linear combination of elements of the form
	\begin{equation}\label{eq:spanning}
		a^{(1)}_{-m_1}\dotsm a^{(r)}_{-m_r} \vacuum
	\end{equation}
	with all $m_i \geq 1$.
\end{notation}
The following result is now an immediate translation of the defining equation \eqref{eq:aff_lie_alg_eqn} for affine Lie algebras.
\begin{proposition}\label{prop:am_bn}
	Let $V = V_{\gAffLie}(1,0)$.
	Let $a,b \in \gLie = V_{(1)}$ and let $m,n \in \ZZ$.
	Then
	\[ a_m b_n = b_n a_m + [a,b]_{m+n} + m \delta_{m,-n} \langle a,b \rangle . \]
\end{proposition}
\begin{proof}
	By \eqref{eq:aff_lie_alg_eqn}, we have in $U(\gAffLie)$ that
	\[ a(m)b(n) = b(n)a(m) + \left[a,b\right](m+n)+m\left\langle a,b\right\rangle \delta_{m+n,0}\mathbf{k}. \]
	Applying this identity on any $c\in V$, we get, by \eqref{eq:identify_gLie_and_V}, that
	\begin{multline*}
	 	a_mb_nc = a(m)b(n)c = b(n)a(m)c + \left[a,b\right](m+n)c+m\left\langle a,b\right\rangle \delta_{m+n,0}\mathbf kc \cr= b_n a_mc + [a,b]_{m+n}c + m \delta_{m,-n} \langle a,b \rangle c,
	\end{multline*}
	because $\mathbf{k}$ acts as the scalar $1$ on $V$.
\end{proof}
In fact, the vertex algebra $V = V_{\gAffLie}(1,0)$ is a graded vertex algebra of type~O.
\begin{proposition}
	Let $V = V_{\gAffLie}(1,0)$.
	\begin{enumerate}
		\item The grading on $V$ from \cref{definition:Z_grading_of_V} is a vertex algebra grading.
		\item For all $a,b\in V$ and $n\in\ZZ$, there is some $u\in U(\gAffLie)$ such that $a_nb = ub$.
	\end{enumerate}
	\end{proposition}
	\begin{proof}
		By linearity, it suffices to prove both statements for elements of the form~\eqref{eq:spanning}.
		So assume $a = a^{(1)}_{-m_1}\dots a^{(r)}_{-m_r}\vacuum$ and $b \in V_{(m)}$ with $a^{(1)},\dots,a^{(r)} \in \gLie$, $r,m \in\ZZ_{\geq 0}$ and $m_1,\dots, m_r\in\ZZ$. Write $l := m_1+\dots+m_r$, hence $a\in V_{(l)}$.
		We claim that
		\begin{equation}\label{eq:claim grading}
			a_n b \in V_{(l+m-1-n)} \quad \text{and there is some } u\in U(\gAffLie) \text{ such that } a_nb = ub .
		\end{equation}
		We prove this claim by induction on $r$. For $r=0$, we have $a=\vacuum$, so the claim follows from the vacuum property (\cref{def:vertex_algebra}\ref{item:vacuum_property}) by taking $u\coloneqq \delta_{-1,n} 1$ (where $1$ is the unit in $U(\gAffLie)$).
 
		Assume now that $r>0$ and set $a'\coloneqq a^{(2)}_{-m_2}\dots a^{(r)}_{-m_r}\vacuum$. By the iterate formula (\cref{def:vertex_algebra}\ref{item:iterate_formula}), we have 
		\[ a_nb = \bigl( a^{(1)}_{-m_1} a' \bigr)_n b = \sum_{i=0}^\infty (-1)^i \binom{-m_1}{i} \bigl( a^{(1)}_{-m_1-i}a'_{n+i} - (-1)^{m_1} a'_{n-m_1-i}a_i^{(1)} \bigr) b. \]
		By the truncation condition (\cref{def:vertex_algebra}\ref{item:truncation_condition}), this reduces to a finite sum: there is an $I$ such that for all $i>I$, we have both $a'_{n+i}b=0$ and $a^{(1)}_ib=0$. By the induction hypothesis, each term lies in $V_{(l+m-1-n)}$ and there are $u^{(i)}$ and $v^{(i)}$ in $U(\gAffLie)$ such that
		\begin{align*}
			a^{(1)}_{-m_1-i}a'_{n+i}b &= a^{(1)}_{-m_1-i}u^{(i)}b, & a'_{n-m_1-i}a^{(1)}_ib = v^{(i)}a^{(1)}_ib.
		\end{align*}
		Now set
		\[ u \coloneqq \sum_{i=0}^I (-1)^i \binom{m_1}{i} \bigl( a^{(1)}(-m_1-i)u^{(i)} - (-1)^{m_1} v^{(i)}a^{(1)}(i) \bigr) \in U(\gAffLie). \]
		Then it is clear that $a_nb = ub$.
		This proves the claim \eqref{eq:claim grading}.
	\end{proof}
	\begin{corollary}\label{cor:left_mult}
		Any $\gAffLie$-submodule of $V_{\gAffLie}(1,0)$ is a left ideal.
	\end{corollary}
	\begin{definition}\label{definition:L}
	The algebra $V_{\gAffLie}(1,0)$ is a graded vertex algebra of type O, so by \cref{lemma:simple_vertex_algebra}, it has a unique maximal graded ideal, which we denote by $I_{\gAffLie}(1,0)$.
	Moreover, the quotient
	\[ L_{\gAffLie}(1,0) \coloneqq V_{\gAffLie}(1,0)/I_{\gAffLie}(1,0). \]
	is a simple graded vertex algebra.
\end{definition}
\begin{remark}
	This grading on $V_{\gAffLie}(1,0)$ induces (non-associative) algebra structures and bilinear forms on the components: the $(n-1)$-st product defines a product on $V_{\gAffLie}(1,0)_{(n)}$ and the $(2n-1)$-st product defines a bilinear form on $V_{\gAffLie}(1,0)_{(n)}$, because $V_{\gAffLie}(1,0)_{(0)}$ is one-dimensional.
	For $n=1$, we recover the Lie bracket and the bilinear form $\kappa$; see \cref{lemma:comp}\ref{lemma:comp:lie_bracket} and \ref{lemma:comp:form} below.
\end{remark}
%
%
	
We collect some specific computations that we will need in \cref{section:results}.
\begin{lemma}\label{lemma:comp}
	Let $V = V_{\gAffLie}(1,0)$ and let $a,b,c \in \gLie = V_{(1)}$.
	Then
	\begin{enumerate}
		\item\label{lemma:comp:lie_bracket} $a_0b = [a,b].$
		\item\label{lemma:comp:form} $a_1b = \langle a,b\rangle\vacuum.$
		\item\label{lemma:comp:a0} $a_0 b \1 c = b \1 [a,c] + [a,b] \1 c$.
		\item\label{lemma:comp:a1} $a_1 b \1 c = [[a,b], c] +  \langle a,b \rangle c + \langle a,c \rangle b$.
		\item\label{lemma:comp:a2} $a_2 b \1 c =  \bigl\langle [a,b], c \bigr\rangle \vacuum$.
		\item\label{lemma:comp:a2sym} $a_2 b \1 b = 0$.
		\item\label{lemma:comp:a0a0} $a_0 a_0 b \1 b = 2 [a,b] \1 [a,b] + b \1 [a,[a,b]] + [a,[a,b]] \1 b$.
		\item\label{lemma:comp:a-1a1} $a \1 a_1 b \1 b = a \1 [[a,b],b] + 2  \langle a,b \rangle a \1 b$.
		\item\label{lemma:comp:a2a0} $a_2 a_0 b\1 b = 0$.
		\item\label{lemma:comp:a1a1} $a_1 a_1 b\1 b = 2  \langle a,b \rangle^2 \vacuum - \langle [a,b], [a,b] \rangle \vacuum$.
		\item\label{lemma:comp:aa1b} $(a\1 a)_1 b = \bigl[ a,[a,b]\bigr] + 2 \langle a,b\rangle a = b_1a\1 a$. 
		\item\label{lemma:comp:a2b} $(a\2 \vacuum)_1 b = -[a,b]$.
	\end{enumerate}
\end{lemma}
\begin{proof}
	For each statement, the strategy is similar:
	\begin{itemize}[itemsep=0ex]
		\item Move operators $a_m$ with $m \geq 0$ to the right by rewriting $a_m b_n$ using \cref{prop:am_bn};
		\item Cancel terms $a^{(1)}_{m_1} \dotsm a^{(i)}_{m_i} b$ whenever $m_1 + \dots + m_i \geq 2$;
		\item Use the creation property $a \1 \vacuum = a$ and $a_n \vacuum = 0$ for $n\geq 0$;
		\item Apply \cref{lemma:comp:lie_bracket,lemma:comp:form} if necessary;
		\item Invoke $\gLie$-invariance (\cref{def:Lie_bilinear_form}) if necessary;
		\item Use the iterate formula (\cref{def:vertex_algebra}\ref{item:iterate_formula}) if necessary, and truncate using $a_nb = 0$ whenever $n \geq 2$. 
	\end{itemize}
	\begin{enumerate}
		\item $a_0b = a_0b_{-1}\vacuum = b_{-1}a_0\vacuum + [a,b]_{-1}\vacuum = [a,b].$
		\item $a_1b = a_1b_{-1}\vacuum = b_{-1}a_1\vacuum + [a,b]_{0}\vacuum + \langle a,b \rangle \vacuum= \langle a,b \rangle \vacuum.$
		\item $a_0 b \1 c = (b \1 a_0 + [a,b]\1) c = b \1 a_0 c + [a,b]\1 c = b \1 [a,c] + [a,b] \1 c$.
		\item $a_1 b \1 c = (b \1 a_1 + [a,b]_0 + \langle a,b \rangle) c = \langle a,c \rangle b + [[a,b], c] + \langle a,b \rangle c$.
		\item $a_2 b \1 c = (b \1 a_2 + [a,b]_1) c = \langle [a,b], c \rangle \vacuum$.
		\item By $\gLie$-invariance, $\big\langle [a,b], b \big\rangle = \big\langle a, [b, b] \big\rangle = 0$, so this follows from \ref{lemma:comp:a2} with $b=c$.
		\item We apply $a_0$ on \ref{lemma:comp:a0} with $b=c$ and we get, by two more applications of \ref{lemma:comp:a0},
			\begin{align*}
				a_0 a_0 b \1 b
				&= a_0 \bigl( b \1 [a,b] + [a,b] \1 b \bigr) \\
				&= \bigl( b \1 [a, [a,b]] + [a,b]\1[a,b] \bigr) + \bigl( [a,b]\1[a,b] + [a,[a,b]]\1b \bigr) .
			\end{align*}
		\item We apply $a\1$ on \ref{lemma:comp:a1} with $b=c$.
		\item 
			Since $a_0 a_2 = a_2 a_0$, this follows from \ref{lemma:comp:a2sym}.
		\item We apply $a_1$ on \ref{lemma:comp:a1} with $b=c$ and we apply $\gLie$-invariance to get
			\begin{align*}
				a_1 a_1 b\1 b
				&= a_1 [[a,b],b] + 2  \langle a,b \rangle a_1 b \\
				&=  \bigl\langle a, [[a,b], b] \bigr\rangle \vacuum + 2  \langle a,b \rangle \langle a,b \rangle \vacuum \\
				&= 2  \langle a,b \rangle^2 \vacuum -  \bigl\langle [a,b], [a,b] \bigr\rangle \vacuum .
			\end{align*}
		\item Using the iterate formula, we get $(a\1 a)_1 b = a_0 a_0 b + 2 a\1 a_1 b = \bigl[ a,[a,b]\bigr] + 2 \langle a,b\rangle a$. By \ref{lemma:comp:a1}, this is also equal to $b_1 a\1 a$.
		\item Using the iterate formula and the vacuum property, we get $(a\2 \vacuum)_1 b = -\vacuum\1 a_0 b = -[a,b]$.
		\qedhere
	\end{enumerate}
\end{proof}
\begin{lemma}\label{lemma:comp_aabb}
	Let $V = V_{\gAffLie}(1,0)$ and let $a,b \in \gLie = V_{(1)}$.
	Then
	\[ a \1 [[a,b],b] - [[a,b],b] \1 a + b \1 [a,[a,b]] - [a,[a,b]] \1 b = 0  . \]
\end{lemma}
\begin{proof}
	By \cref{prop:am_bn}, we have $a \1 b \1 - b \1 a \1 = [a,b] \2$, so by applying this on~$\vacuum$, we get
	\[ a \1 b - b \1 a = [a,b] \2 \vacuum . \]
	Hence
	\begin{multline*}
		a \1 [[a,b],b] - [[a,b],b] \1 a + b \1 [a,[a,b]] - [a,[a,b]] \1 b \\
			= \bigl( \bigl[ a, [[a,b],b] \bigr] + \bigl[ b, [a, [a,b]] \bigr] \bigr) \2 \vacuum .
	\end{multline*}
	By the Jacobi identity, the right hand side is equal to $\bigl[ [b,a], [a,b] \bigr] \2 \vacuum = 0$.
\end{proof}

\section{Constructing $A(\gLie, \kappa)$ inside $L_{\gAffLie}(1,0)$}\label{section:results}

We are now prepared to reconstruct the Chayet--Garibaldi algebra $A(\gLie, \kappa)$ inside the simple graded vertex algebra $L_{\gAffLie}(1,0)$.
Recall \cref{definition:L} and let us fix the notation $V$, $I$ and $L$ for $V_{\gAffLie}(1,0)$, $I_{\gAffLie}(1,0)$ and $L_{\gAffLie}(1,0)$, respectively.

It will turn out that the Chayet--Garibaldi algebra can be obtained by restricting the degree two component $L_{(2)}$ of $L$ to an appropriate subspace
\[ L\twosym \coloneqq \left\langle a\1b+b\1a\mid a,b \in L_{(1)}\right\rangle \leq L_{(2)} \]
and endowing it with the ``Jordan product'' $d \bullet e := \tfrac{1}{2} (d_1e + e_1d)$ and with the bilinear form $(d,e) \mapsto d_3e$.

Before presenting the details, we give an outline of the procedure for the proof.
Notice that, even though the construction can be obtained directly from the simple graded vertex algebra $L$, the proof will be performed completely inside the original affine vertex algebra $V$; only in the final stage of the proof, we will pass to the quotient~$L$.
\begin{enumerate}[itemsep=0ex]
	\item First, we relate the symmetric square $\symsq \gLie$ with the subspace $V\twosym$ of $V_{(2)}$. We show that the product $*$ defined in  \eqref{formula:star} corresponds to the Jordan product obtained from the product $(d,e) \mapsto d_1e$ on $V_{(2)}$ (\cref{prop:theta_is_isomorphism}) and that the bilinear form $\tau$ defined in \cref{definition:tau} corresponds to the map $(d,e) \mapsto d_3e$ on $V\twosym$ (\cref{prop:tau}).
	\item Next, we interpret $\ker S$---see \eqref{def:chga_S}---in the context of $V_{(2)}$ (\cref{lemma:kernels_are_equal}). At this point, we are able to construct the algebra $A(\gLie, \kappa)$ from the vertex algebra $V$, without having to define $*$ or $S$ explicitly (\cref{cor:A_isom_quotient}).
	\item We then show that $\ker S$ generates a proper ideal of $V_{(2)}$ (\cref{lemma:ideal_ker_T}). Consequently, this ideal coincides with the intersection of the maximal graded ideal $I$ with $V_{(2)}$ (\cref{cor:kerT_and_maximal_ideal}).
	\item We finally take the quotient $L$ of $V$ by this maximal graded ideal $I$. The main result now follows by combining all previous steps (\cref{theorem:main_result}).
\end{enumerate}

\medskip

We will now go through the different steps of this process in detail.
We continue to impose \cref{assumptions:lie_algebras}.
We also continue to identify $\gLie$ with $V_{(1)}$ (as vector spaces) via the map $\gLie \to V_{(1)} \colon a \mapsto a(-1)\vacuum$.

\begin{lemma}\label{lemma:basis_V2}
	Let $\{ a^{(i)} \mid i\in \indexset \}$ be a totally ordered basis for $\gLie$. Then the set
	\[ \mathcal B \coloneqq \bigl\{ {a^{(i)}}\!\2\vacuum\mid i\in \indexset \bigr\} \ \cup\  \bigl\{{a^{(i)}}\!\1a^{(j)} + {a^{(j)}}\!\1a^{(i)} \mid i, j \in \indexset, i \geq j \bigr\} \]
	is a basis for $V_{(2)}$.
\end{lemma}
\begin{proof}
	For all $i, j \in \indexset$, we let
	\[ b_i \coloneqq {a^{(i)}}\!\2\vacuum \qquad \text{and} \qquad c_{ij} \coloneqq {a^{(i)}}\!\1a^{(j)} . \]
	By \cref{definition:Z_grading_of_V}, the space $V_{(2)}$ has a basis
	\[ \mathcal B' = \{ b_i \mid i\in \indexset\}\cup \{ c_{ij} \mid i,j \in \indexset, i \geq j \}. \]
	By \cref{prop:am_bn}, we have
	\[ \bigl[ a^{(i)},a^{(j)} \bigr]\2 = {a^{(i)}}\!\1a^{(j)}\!\1 - {a^{(j)}}\!\1a^{(i)}\!\1 , \]
	hence
	\[ 2c_{ij} =  c_{ij} + c_{ji} + \bigl[ a^{(i)},a^{(j)} \bigr]\2\vacuum . \]
	Since $\bigl[ a^{(i)},a^{(j)} \bigr]\2\vacuum \in \langle b_i \mid i \in \indexset \rangle$, it follows that also $\mathcal B$ is a basis for $V_{(2)}$.
\end{proof}
\begin{definition}
	We define the subspace
	\[ V\twosym \coloneqq \left\langle a\1b+b\1a\mid a,b \in \gLie = V_{(1)}\right\rangle \leq V_{(2)} . \]
\end{definition}
\begin{lemma}\label{lemma:V_2_direct_sum_decomposition}
	We have a decomposition $V_{(2)} = \mathcal D V_{(1)} \oplus V\twosym.$
\end{lemma}
\begin{proof}
	Since $\mathcal Da^{(i)} = {a^{(i)}}\!\2\vacuum$ for all $a \in \gLie = V_{(1)}$, the result follows from \cref{lemma:basis_V2}.
\end{proof}
We can now identify the symmetric square $\symsq \gLie$ with $V\twosym$.
\begin{definition}
	We define a $k$-linear map
	\[
		\thetaAlt \colon \symsq \gLie \to V\twosym \colon 4ab \mapsto a\1b+b\1a .
	\]
	Since $\thetaAlt$ maps the basis $\{ a^{(i)}a^{(j)} \mid i,j\in\indexset, i \geq j \}$ of $\symsq \gLie$ to a scalar multiple of the basis $\bigl\{{a^{(i)}}\!\1a^{(j)} + {a^{(j)}}\!\1a^{(i)} \mid i, j \in \indexset, i \geq j \bigr\}$ of $V\twosym$,
	it is an isomorphism (of vector spaces).
\end{definition}
\begin{remark}
	Because of the factor $4$ in the definition of $\thetaAlt$, we have $\thetaAlt(aa) = \frac12a_{-1}a$ for all $a \in \gLie$.
	By \cref{lemma:comp}\ref{lemma:comp:aa1b}, we then have
	\begin{equation}\label{eq:action_V2_on_V1}
		\thetaAlt(aa)_1b = \tfrac12(a_{-1}a)_1b = \tfrac12S(aa)b
	\end{equation}
	for all $a,b\in\gLie$.
	(Recall the definition of $S$ from \eqref{def:chga_S} in \cref{definition:Ag}.)
\end{remark}
\begin{definition}\label{def:T}
	We define the map	
	\[ \Talt \colon V_{(2)}\to \End(V_{(1)}) \colon d \mapsto 2d_1 \restrict{V_{(1)}}.\]
\end{definition}	
\begin{lemma}
	The map $S$ from \cref{definition:Ag} coincides with $\Talt \circ \thetaAlt$.
\end{lemma}
\begin{proof}
	This follows from the definitions of $S$ and $\Talt$ together with \eqref{eq:action_V2_on_V1}.
\end{proof}
It is now natural to translate the commutative product $*$ on $\symsq\gLie$ (see \eqref{formula:star}) to $V_{(2)}$.
The somewhat complicated product $*$ has a surprisingly simple interpretation in $V_{(2)}$.
Since the vertex algebra product $_1$ is not commutative, it is natural to make it into a commutative product via the ``Jordan product''.
\begin{definition}\label{def:bullet}
	We define a product $\bulletAlt$ on $V_{(2)}$ by setting
	\[ d \bulletAlt e \coloneqq \tfrac1{2}(d_1e + e_1d) \]
	for all $d,e \in V_{(2)}$.
\end{definition}
\begin{proposition}\label{prop:theta_is_isomorphism}
	The map $\thetaAlt$ is an algebra isomorphism from $(\symsq \gLie,*)$ to $(V\twosym,\bulletAlt)$.
	In particular, $V\twosym$ is a subalgebra of $V_{(2)}$ for the product $\bulletAlt$. 
\end{proposition}
\begin{proof}	
	Let $a,b \in \gLie = V_{(1)}$ be arbitrary.
	By the iterate formula (\cref{def:vertex_algebra}\cref{item:iterate_formula}) and the fact that $a_i V_{(2)} = 0$ whenever $i \geq 3$ (by \eqref{eq:graded}),
	we have
	\[ (a\1a)_1(b\1b)
		= \sum_{i\geq 0}\left(a_{-1-i}a_{1+i} + a_{-i}a_i\right)b\1b
		= (a_0a_0+2a\1a_1+2a\2a_2)b\1b . \]
	By \cref{lemma:comp}\cref{lemma:comp:a0a0,lemma:comp:a-1a1,lemma:comp:a2sym}, this implies
	\begin{multline*}
		(a\1a)_1(b\1b) \\
		= 4 \langle a,b \rangle a\1b + 2a\1\bigl[[a,b],b\bigr] + \bigl[a,[a,b]\bigr]\1b + b\1\bigl[a,[a,b]\bigr] + 2[a,b]\1[a,b].
	\end{multline*}
	Hence
	\begin{multline*}
		(a\1a)_1(b\1b) + (b\1b)_1(a\1a) \\
		= 4 \langle a,b \rangle (a\1b + b\1a) + 4[a,b]\1[a,b] \hspace*{14ex} \\
		 + 3a\1\bigl[[a,b],b\bigr] + \bigl[[a,b],b\bigr]\1a + 3b\1\bigl[a,[a,b]\bigr] + \bigl[a,[a,b]\bigr]\1b .
	\end{multline*}
	By \cref{lemma:comp_aabb,def:bullet}, we can also write this as
	\begin{multline*}
		2 (a\1a) \bulletAlt (b\1b) 
		= 4\langle a,b \rangle (a\1b + b\1a) + 4 [a,b]\1[a,b] \\
		 + 2 \bigl( a\1\bigl[[a,b],b\bigr] + \bigl[[a,b],b\bigr]\1a + b\1\bigl[a,[a,b]\bigr] + \bigl[a,[a,b]\bigr]\1b \bigr) .
	\end{multline*}
	Hence
\begin{align*} 8\thetaAlt(aa)\bulletAlt \thetaAlt(bb)
	& = 8\bigl\langle a,b \bigr\rangle \thetaAlt(2ab)
	+ 4\thetaAlt\bigl(2a\bigl[[a,b],b\bigr]\bigr) 
	+ 4\thetaAlt\bigl(2\bigl[a,[a,b]\bigr]b\bigr) + 8\thetaAlt([a,b][a,b]) \\
	& = 8\thetaAlt \Bigl( 2\langle a,b \rangle ab
	+ a\bigl[[a,b],b\bigr] 
	+ \bigl[a,[a,b]\bigr]b + [a,b][a,b]\Bigr)\\
	& = 8\thetaAlt(aa*bb).\qedhere
	\end{align*}
\end{proof}
\begin{remark}\label{remark:connection_bor}
	The product $\bulletAlt$ in characteristic $0$ already appeared in \cite[Section~9, last paragraph]{Bor86}. It is denoted by $\times_0$ and is defined (on the whole vertex algebra $V$) as the alternating sum
	\[ a\times_0 b \coloneqq \sum_i\frac{(-1)^i}{i+1}\D^{(i)}a_{i+1}b . \]
	If we specialize this to $a,b\in V_{(2)}$, then the terms with $i\geq2$ are zero. So, for $a,b\in V_{(2)}$ we have, using \cref{def:vertex_algebra}\cref{item:skew_symmetry}, that
	\[ a \times_0 b = a_1b - \tfrac12\D a_2 b = \tfrac12 (a_1b + b_1a) = a\bulletAlt b. \]

\end{remark}

We now focus on the bilinear form $\tauAlt$ introduced in \cref{definition:tau}.
\begin{proposition}\label{prop:tau}
	The bilinear form $\tauAlt$ on $A(\gLie, \kappa)$ corresponds to the map
	\[  V\twosym \times V\twosym \to k \vacuum \colon (d,e) \mapsto d_3e . \]
\end{proposition}
\begin{proof}
	We have to verify that
	\[ \thetaAlt(aa)_3\thetaAlt(bb) = \tauAlt\bigl( S(aa), S(bb) \bigr) \vacuum \]
	for all $a,b \in \gLie$.
	Recall from \cref{definition:tau} that $\tauAlt\bigl( S(aa), S(bb) \bigr) = \frac12 \langle S(aa)b, b \rangle$.
	By the iterate formula (\crefiterateformula) and the fact that $a_i V_{(2)} = 0$ whenever $i \geq 3$ (by~\eqref{eq:graded}),
	we have
	\[ (a\1a)_3(b\1b)
		= \sum_{i\geq 0}\left(a_{-1-i}a_{3+i} + a_{2-i}a_i\right)b\1b
		= (a_2a_0+a_1a_1+a_0a_2)b\1b . \]
	It follows from \cref{lemma:comp}\cref{lemma:comp:a2a0,lemma:comp:a1a1,lemma:comp:a2sym} that
	\begin{align}
		(a\1a)_3(b\1b)
		&= 2\langle a,b\rangle^2\vacuum + \left\langle b,\bigl[a,[a,b]\bigr] \right\rangle \vacuum \\
		&= \left\langle b,\bigl[a,[a,b]\bigr]+2\langle a,b\rangle a\right\rangle\vacuum\\
		&= \langle b,2 S(aa)b\rangle\vacuum\\
		&= 4\tauAlt(S(aa), S(bb))\vacuum.
		\qedhere
	\end{align}
\end{proof}

Our next step is to proceed to the quotient vertex algebra $L= L_{\gAffLie}(1,0)$ and to construct the algebra $A(\gLie, \kappa)$ directly as a \emph{subalgebra} of $L$.
This is where we will need the assumptions that $\gLie$ has trivial center and that $\langle\cdot,\cdot\rangle$ is non-degenerate (\cref{assumptions:lie_algebras}).
We thank Jari Desmet for the proof of the next lemma.
\begin{lemma}\label{lemma:kernels_are_equal}
	We have $\thetaAlt(\ker S) = \ker \Talt$.
\end{lemma}
\begin{proof}
	Observe that $S=\Talt\circ \thetaAlt$ implies $\thetaAlt(\ker S) \subseteq \ker \Talt$. By \cref{lemma:V_2_direct_sum_decomposition}, we have $V_{(2)} = V\twosym\oplus\D V_{(1)}$. Our goal is to show that $\ker \Talt \leq V\twosym$.
	Let $a,b,c\in \gLie$.
	By \cref{lemma:comp}\cref{lemma:comp:aa1b,lemma:comp:a2b} we have
	\begin{align}
		\Talt(a\1a)b &= 2\bigl[a, [a,b]\bigr] + 4 \langle a,b \rangle, \label{eq:T1} \\
		\Talt(a\2\vacuum)b &= - 2[a,b]. \label{eq:T2}
	\end{align}
	By \eqref{eq:T1}, we have
	\[
		\bigl\langle \Talt (a\1 a)b,c \bigr\rangle 
		= \left\langle 2\bigl[a,[a,b]\bigr] + 4\langle a,b \rangle a,c\right\rangle
		= \left\langle b, 2\bigl[a,[a,c]\bigr] + 4\langle a,c \rangle a\right\rangle = \bigl\langle b,\Talt (a\1a)c \bigr\rangle,
	\]
	so for each $d\in V\twosym$, the linear operator $\Talt(d)\in \End V_{(1)}$ is \emph{symmetric} (see \cref{definition:symmetric_map}).
	On the other hand, by \eqref{eq:T2},
	\[
		\bigl\langle \Talt(a\2\vacuum)b,c\bigr\rangle = - \bigl\langle 2[a,b],c \bigr\rangle
			= \bigl\langle b,2[a,c] \bigr\rangle = \bigl\langle b,\Talt(a\2\vacuum)c\bigr\rangle,
	\]
	so for each $a\2\vacuum\in \D V_{(1)}$, the operator $\Talt(a\2\vacuum)$ is \emph{antisymmetric}.
	
	Now let $d + e \in \ker \Talt$ with $d \in V\twosym$ and $e \in \D V_{(1)}$. Then $\Talt(d) = - \Talt(e)$, so the linear operator $\Talt(e)$ is both symmetric and antisymmetric with respect to the \emph{non-degenerate} bilinear form $\kappa = \langle \cdot, \cdot \rangle$.
	This is only possible when $\Talt(e)=0$.
	However, because $Z(\gLie) = 0$, it now follows from \eqref{eq:T2} that $e = 0$.
	We conclude that $\ker \Talt \leq V\twosym$.
\end{proof}
\begin{corollary}\label{cor:A_isom_quotient}
	We have an algebra isomorphism
	\[ A(\gLie, \kappa) \cong V\twosym/\ker \Talt  . \]
\end{corollary}
Now we prove that $\ker \Talt$ is contained in the unique maximal graded ideal~$I$.

\begin{lemma}\label{lemma:ideal_ker_T}
	The ideal of $V$ generated by the subspace $\ker \Talt \leq V_{(2)}$ is contained in $\bigoplus_{n=2}^\infty V_{(n)}$.
	In particular, it is a proper graded ideal.
\end{lemma}
\begin{proof}
	Let $J$ be the left ideal of $V$ generated by $\ker \Talt$.
	By \cref{cor:left_mult}, we have $J = U(\gAffLie)\cdot\ker \Talt$.
	By \cref{lemma:J_ideal}, the subspace
	\[ J' \coloneqq \sum_{i=0}^\infty \D^{(i)}J \]
	of $V$ is then precisely the two-sided ideal generated by $\ker \Talt$.
	Our goal is to show that $J'$ is a proper ideal.
	
	Our first and most important step is to show that $J \subseteq \bigoplus_{n=2}^\infty V_{(n)}$.
	We claim that
	\begin{equation}\label{eq:a_n_ker}
		a_n d \in \ker \Talt \qquad \text{for all } a \in \gLie, \ d \in \ker \Talt, \ n \in \ZZ_{\geq 0} .
	\end{equation}
	For $n > 2$, this is obvious from the grading.
	For $n = 2$, we use the fact that $\ker \Talt \leq V\twosym$ (by \cref{lemma:kernels_are_equal}) together with the fact that $a_2d = 0$ for all $d \in V\twosym$ (by \cref{lemma:comp}\cref{lemma:comp:a2sym}).
	For $n = 1$, we use the observation made in \cref{def:T} that $a_1 d = \frac12\Talt(d)(a) = 0$ for all $d \in \ker \Talt$.
	Finally, for $n = 0$, we note that by the iterate formula (\crefiterateformula), we have $(a_0 d)_1 = a_0 d_1 - d_1 a_0$.
	In particular, if $d \in \ker \Talt$, then
	\[ \Talt(a_0 d)(b) =  2(a_0 d)_1 b =  2(a_0 d_1 - d_1 a_0) b = \bigl( [a, \Talt(d)(b)] - \Talt(d)([a,b]) \bigr) = 0 \]
	for all $b \in \gLie$, hence also $a_0 d \in \ker \Talt$.
	This proves our claim \eqref{eq:a_n_ker}.

	By the PBW theorem (\cref{theorem:pbw}) with respect to the totally ordered basis for $\gAffLie$ as in \cref{corollary:basis_for_V}, $U(\gAffLie) \cdot \ker \Talt$ is spanned by elements of the form
	\[ {a^{(1)}}_{n_1}\dotsm {a^{(r)}}_{n_r} \, c \quad \text{ with } a^{(1)},\dots, a^{(r)}\in \gLie,\ c\in \ker \Talt\ \text{and}\ n_1\leq n_2 \leq \dots \leq n_r \in \ZZ.\]
	Together with our previous claim \eqref{eq:a_n_ker}, this implies that $U(\gAffLie)\cdot \ker T$ is spanned by elements of the form 
	\[ {a^{(1)}}_{n_1}\dotsm {a^{(r)}}_{n_r} \, c \quad \text{ with } a^{(1)},\dots, a^{(r)}\in \gLie,\ c\in \ker \Talt\ \text{and}\ n_1\leq n_2 \leq \dots \leq n_r < 0 . \]
	By the grading, this now implies that indeed $J = U(\gAffLie)\cdot \ker \Talt \subseteq \bigoplus_{n=2}^\infty V_{(n)}$.
	
	It is now easy to pass to $J' = \sum_{i=0}^\infty \D^{(i)}J$.
	Indeed, let $m,n \in \ZZ$ with $m \geq 0$ and let $a \in V_n$; then by the grading again, $\D^{(m)}a = a_{(-m-1)}\vacuum \in V_{(n+m)}$.
	It follows that $\D^{(m)}J\subseteq \bigoplus_{n=2+m}^\infty V_{(n)}$, so indeed $J'\subseteq \bigoplus_{n=2}^\infty V_{(n)}$.
\end{proof}
Recall that we write $I \coloneqq I_{\gAffLie}(1,0)$ for the unique maximal graded ideal in $V_{\gAffLie}(1,0)$.
\begin{corollary}\label{cor:kerT_and_maximal_ideal}
	We have $I\cap V_{(2)} = \ker \Talt$.
\end{corollary}
\begin{proof}
	By \cref{lemma:ideal_ker_T}, it is sufficient to show that $I\cap V_{(2)}\subseteq \ker \Talt$.
	Suppose the contrary; then it follows from \cref{def:T} that there is an $a\in I\cap V_{(2)}$ and $b\in V_{(1)}$ such that $a_1b\neq0$.
	The form $\kappa = \langle \cdot ,\cdot \rangle$ is non-degenerate, however, so there exists some $c\in V_{(1)}$ such that $\langle c, a_1b \rangle \neq 0$. Then $\vacuum = \frac{1}{\langle c, a_1b \rangle} c_1a_1b \in I$, which is a contradiction.
\end{proof}

We are now ready to state our Main Theorem (see also page \pageref{mainthm}).
\begin{theorem}\label{theorem:main_result}
	Let $\gLie$ be a Lie algebra over a field $k$ with $\Char(k) \neq 2$ with $Z(\gLie) = 0$, let $\kappa$ be a non-degenerate invariant symmetric bilinear form on $\gLie$.

	Then $L\twosym$ equipped with the product
	\[ a \bulletAlt b \coloneqq \tfrac{1}{2}(a_1b + b_1a) \]
	and the bilinear form
	\[ (a, b) \mapsto a_3 b ,\]
	is isomorphic to the Chayet--Garibaldi algebra $A(\gLie, \kappa)$ (as an algebra with bilinear form).
\end{theorem}
\begin{proof}
	By definition, we have $L = V/I$, hence $L\twosym \cong V\twosym/(I\cap V\twosym)$.
	The result now follows immediately from \cref{cor:A_isom_quotient,cor:kerT_and_maximal_ideal,prop:tau}.
\end{proof}

\begin{remark}\label{remark:unital}
	The vertex algebra $L = V/I$ is a vertex \emph{operator} algebra if and only if the corresponding algebra $A = A(\gLie, \kappa)$ is \emph{unital}. In this case, the conformal vector $\omega$ corresponds to the unit in $A$ under the isomorphism of \cref{theorem:main_result}, and the central charge is $4\tau(\Id_\gLie,\Id_\gLie)$.
	
	We briefly explain why this is the case, leaving the details to the reader.
	We rely on \cite{LM18} for the setup of vertex operator algebras over arbitrary fields $k$ with $\Char(k) \neq 2$; see, in particular, its Definition 5.11 for the definition of a vertex operator algebra. In particular, we use the standard notation $L(n)$ for the modes of the conformal vector $\omega$.
	
	\medskip
	
	Assume first that $L$ is a vertex operator algebra with conformal vector $\omega \in L_{(2)}$.
	Then $\omega_1$ acts on each $L_{(n)}$ as multiplication by $n$, so in particular $\omega_1|_{L_{(1)}} = \Id$.
	Using \cref{lemma:V_2_direct_sum_decomposition} and $L(1) \omega = 0$, it follows that $\omega \in L\twosym$.
	It is then not hard to show that under the isomorphism from \cref{theorem:main_result}, $\omega$ corresponds to the unit of $A$.
	
	\medskip
	
	Conversely, assume that $A$ is unital; by \cref{lemma:unit}, the the unit is $\Id_\gLie \in \im S$, so we can write $\Id_\gLie = \sum_iS(u^{(i)}u^{(i)})$.
	Let $\omega \in L\twosym$ be the element corresponding to $\Id_\gLie \in A$ under the isomorphism from \cref{theorem:main_result}.
	To show that $L$ is a vertex operator algebra with conformal vector $\omega$, we can follow the procedure from \cite[Theorem~6.2.16]{LL04} (the so-called Segal--Sugawara construction, see e.g., \cite[\S 2.5.10 and \S 3.4.8]{FBZ04}), but of course our setup is different. While \cite[Theorem 6.2.16]{LL04} starts from the assumption that the Casimir operator $\Omega$ acts as a scalar, we start from the Chayet--Garibaldi algebra~$A$. Just as in the proof of \lc, it suffices to show that $a_n \omega = \delta_{n,1} a$ for all $a \in \gLie$ and all $n \geq 0$ (this is (6.2.51) in \lc).
	The proof of (6.2.52) holds \emph{mutatis mutandis}, but (6.2.53) and (6.2.54) require a different ``operator product expansion''-style computation, using \cref{lemma:comp}\cref{lemma:comp:a0,lemma:comp:aa1b} together with the fact that $a_2\omega = 0$. 
	The rest of the proof then again holds \emph{mutatis mutandis}, showing that $L$ is indeed a vertex operator algebra of with conformal vector $\omega$. Only the computation of the central charge (6.2.58) requires a different approach using \cref{prop:tau}.
\end{remark}

\end{document}